\documentclass{amsart}
\usepackage{amsmath}
\usepackage{amssymb}
\usepackage{graphicx}
\usepackage{mathabx}
\usepackage{tikz}

\newcommand{\tikzmark}[2][]{\tikz[remember picture, overlay]\node[#1,inner sep=0pt,xshift=-\tabcolsep](#2){};\ignorespaces}
\newcommand{\tm}[2][]{\tikzmark[yshift=#1]{#2}}

\newcommand\tboldsymbol[1]{%
\protect\raisebox{0pt}[0pt][0pt]{%
$\underset{\widetilde{}}{\boldsymbol{#1}}$}\mbox{\hskip 1pt}}
\newcommand{\ca}[1]{\ensuremath{\mathcal{#1}}}
\newcommand{\n}{\ensuremath{n \in \nat}}
\newcommand{\set}[2]{\ensuremath{\{#1 \hspace{0.3mm} \mid \hspace{0.3mm} #2\}}}
\newcommand{\ie}{{\rm i.e.,} }
\newcommand{\N}{\ensuremath{\mathbb N}}
\newcommand{\lh}{{\rm lh}}
\newcommand{\R}{\ensuremath{\mathbb R}}
\newcommand{\ep}{\ensuremath{\varepsilon}}

\newcommand{\nat}{\mathbb{N}}
\newcommand{\pair}[1]{[#1]}
\newcommand{\bolds}{\tboldsymbol{\Sigma}}
\newcommand{\boldp}{\tboldsymbol{\Pi}}
\newcommand{\cantor}{2^{\nat}}
\newcommand{\baire}{\nat^\nat}
\renewcommand{\d}{{\rm dpth}}
\renewcommand{\l}{{\rm lvl}}
\newcommand{\consymbol}{\ast}
\newcommand{\cn}[2]{#1\consymbol#2}

\newcommand{\norm}[1]{\|#1\|}

\newtheorem{theorem}{Theorem}

\newtheorem{corollary}[theorem]{Corollary}

\begin{document}


\title{Intersections of $\ell^p$ spaces in the Borel hierarchy}

\author{Vassilios Gregoriades}

\thanks{The author would like to thank Vassili Nestoridis for helpful discussions.}

\address{National Technical University of Athens\\
School of Applied Mathematical and Physical Sciences\\
9, Iroon Polytechniou St, Athens\\
Greece, Postal Code 157 80}

\email{vgregoriades@math.ntua.gr}


\begin{abstract}
We show that if $Y$ is one of the spaces $\ell^q$, $c_0$, $\ell^\infty$ or ${\textstyle \bigcap_{p > b}} \ell^p$ where $0 < q,b < \infty$, and the Fr\'{e}chet space $\textstyle \bigcap_{p > a} \ell^p$ is contained in $Y$ properly, then $\textstyle \bigcap_{p > a} \ell^p$ first shows up in the Borel hierarchy of $Y$ at the multiplicative class of the third level. In particular $\textstyle \bigcap_{p > a} \ell^p$ is neither an $F_\sigma$ nor a $G_\delta$ subset of $Y$. This answers a question by Nestoridis. This result provides a natural example of a set in the third level of the Borel hierarchy and with its help we also give some examples in the fourth level.
\end{abstract}


\date{\today}

\maketitle


\section{Introduction and Results}

It is frequent in analysis to encounter $\subseteq$-chains of topological vector spaces $(X_i)_{i \in (\ca{I},\preceq)}$ and ask questions about their relation with respect to the structure that they carry. For example the questions whether $X_i$ is a meager $F_\sigma$ subset of $X_j$ (which implies topological genericity) or whether $\left(X_j \setminus X_i\right) \textstyle \ \bigcup \ \{0\}$ contains an infinite dimensional closed subspace (spaceability), where $X_i \subseteq X_j$, have received notable attention; see for example \cite{bayart_grosse-erdmann_nestoridis_papadimitropoulos_abstract_theory_of_universal_series_and_applications,kahane_Baires_category_theorem_and_trigonometric_series,grosse-erdmann_universal_families_and_hypercyclic_operators,gonzalez_pellegrino_seoane-sepulveda_linear_subsets_of_non_linear_sets_in_tvs} and \cite{nestoridis_a_project_about_chains_of_spaces}.

In this article we are concerned with the chain of $\ell^p$ spaces, $p \in (0,\infty)$, together with $c_0$ and $\ell^\infty$ under the topological aspect. As it was shown by Nestoridis the space $\ell^p$ is an $F_\sigma$ meager subset of $\ell^q$, where $0 < p < q < \infty$; in fact the same is true if instead of $\ell^q$ we have one of the following spaces $c_0$, $\ell^\infty$ or ${\textstyle \bigcap_{q > b}} \ell^q$ for $b > p$, see \cite[Proposition 1]{nestoridis_a_project_about_chains_of_spaces}. Consequently, if $Y$ is one of the latter spaces and $a \in [0,\infty)$ is such that ${\textstyle \bigcap_{p > a}} \ell^p$ is contained in $Y$ properly, then the intersection ${\textstyle \bigcap_{p > a}} \ell^p$ is contained in a meager $F_{\sigma}$ subset of $Y$; but it is left open in \cite{nestoridis_a_project_about_chains_of_spaces} (see the comments following the proof of Proposition 1) if the latter intersection remains $F_\sigma$ in $Y$.

We prove that if $Y$ is one of the spaces $\ell^q$, $c_0$, $\ell^\infty$ or ${\textstyle \bigcap_{p > b}} \ell^b$ where $0 < q,b < \infty$, then the intersection $\textstyle \bigcap_{p > a} \ell^p$ first shows up in the Borel hierarchy of $Y$ at the multiplicative class of the third level, where $0 \leq a <q$ if $Y = \ell^q$, $0 \leq a < b$ if $Y={\textstyle \bigcap_{p > b}} \ell^p$, and $0 \leq a <\infty$ if $Y = c_0,\ell^\infty$. In particular ${\textstyle \bigcap_{p > a}} \ell^p$ is not an $F_\sigma$ or a $G_\delta$ subset of $Y$; this answers Nestoridis' question. This result provides a natural example of a set in the third level of the Borel hierarchy and with its help we also give some examples in the fourth level. 

We proceed with a brief review of the necessary notions. First we clarify that we include $0$ in the set of natural numbers and consequently all our sequences (unless stated otherwise) have a $0$-th term. The sequential space $\ell^p$, where $0 < p < \infty$ is the vector space of all real sequences $(x_n)_{\n}$ for which $\sum_{n=0}^\infty |x_n|^p < \infty$. As it is well-known the spaces $\ell^p$, $p > 0$, are increasing on $p$ and in fact for all $0 < p < q$ and all sequences $(x_n)_{\n} \in \ell^p$ we have $\left(\sum_{n=0}^\infty |x_n|^q\right)^{1/q} \leq \left(\sum_{n=0}^\infty |x_n|^p\right)^{1/p}$.\smallskip

When $p \geq 1 $ the space $\ell^p$ admits the norm
\[
\norm{x}_p = \left(\sum_{n=0}^\infty|x_n|^p\right)^{1/p}, \quad x = (x_n)_{\n} \in \ell^p.
\]

Then $(\ell^p, \norm{\cdot}_p)$ is a separable Banach space, $1 \leq p < \infty$.

If $0 < p < 1$ we define the metric $d_p$ on $\ell^p$ by
\[
d_p(x,y) = \sum_{n=0}^\infty |x_n-y_n|^p, \quad x = (x_n)_{\n},  \ y = (y_n)_{\n} \in \ell^p.
\]
Then $(\ell^p,d_p)$ is a complete separable metric space.

We are also concerned with spaces of the form $\textstyle \bigcap_{q > b} \ell^q$ where $0 \leq  b < \infty$.  First we fix once and for all sequences $(p^b_i)_{i \in \nat}$, $b \geq 0$, of positive real numbers such that $p^b_i \downarrow b$; if $b < 1$ we assume that $p^b_0 < 1$. We define the metric $d_{>b}$ on $\textstyle \bigcap_{q > b} \ell^q$ depending on the value of $b$. If $b \geq 1$ we define
\[
d_{q > b}(x,y) = \sum_{i=0}^\infty 2^{-(i+1)} \cdot \dfrac{\norm{x-y}_{p^b_i}}{1+\norm{x-y}_{p^b_i}}, \quad x,y \in \ell^b,
\]
(notice that $p^b_i > 1$ so that $\norm{\cdot}_{p^b_i}$ is defined). 

If $b < 1$ we define 
\[
d_{q > b}(x,y) = \sum_{i=0}^\infty 2^{-(i+1)} \cdot \dfrac{d_{p^b_i}(x,y)}{1+d_{p^b_i}(x,y)}, \quad x,y \in \ell^b,
\]
where $d_{p^b_i}$ is as above (notice that $p^b_i < 1$).\smallskip
 
Then $(\textstyle \bigcap_{q > b} \ell^q,d_{>b})$ is a complete separable metric space.\smallskip

The sequence space $\ell^\infty$ is the vector space of all {\em bounded} real sequences and is equipped with the {\em supremum norm}, $\norm{x}_\infty = \sup\set{|x_n|}{\n}$, $x = (x_n)_{\n} \in \ell^\infty$. Then $(\ell^\infty,\norm{\cdot}_\infty)$ is a non-separable Banach space.\smallskip

The space $c_0$ consists of all real sequences that converge to $0$ and with the (restriction of the) supremum norm $\norm{\cdot}_{\infty}$ it is a separable Banach space.\smallskip

Next we employ some tools from descriptive set theory. Given a metric space $X$ we denote by $\bolds^0_1(X)$ the family of all open subsets of $X$ and by $\boldp^0_1(X)$ the family of all closed subsets of $X$. Recursively we define $\bolds^0_{n+1}(X)$ to be the family of all countable unions of $\boldp^0_n(X)$ sets and $\boldp^0_{n+1}(X)$ to be their complements in $X$. For example $\bolds^0_2(X)$ is the family of all $F_\sigma$ subsets of $X$, $\boldp^0_2(X)$ is the family of all $G_\delta$ subsets of $X$ and so on. Instead of $A \in \bolds^0_n(X)$ we will say that $A$ is a $\bolds^0_n$ subset of $X$, or simpler that $A$ is a $\bolds^0_n$ set when $X$ is understood from the context. The classes $\bolds^0_n$ and $\boldp^0_n$ are also known as the \emph{additive} and the \emph{multiplicative} classes of the $n$-th level in the Borel hierarchy. It is easy to see that the pre-image of a $\bolds^0_n$ subset $B$ of $Y$ under a continuous function $f: X \to Y$ is a $\bolds^0_n$ subset of $X$; obviously the same holds for $\boldp^0_n$.

It is well-known that $\bolds^0_{n}(X) \textstyle \cup \boldp^0_n(X) \subseteq \bolds^0_{n+1}(X) \textstyle \cap \boldp^0_{n+1}(X)$, \ie every $\bolds^0_n(X)$ (and consequently every $\boldp^0_n(X)$) set is both $\bolds^0_{n+1}(X)$ and $\boldp^0_{n+1}(X)$. If $X$ is an uncountable complete separable metric space then $\bolds^0_n(X) \neq \boldp^0_n(X)$. 

Recall that a topological space is {\em Polish} if it generated by a complete separable metric space. It is clear that the preceding definitions can be given in the context of Polish spaces and it is irrelevant which accompanying complete metric we choose.

In the sequel we will employ the \emph{Baire space} $\baire$ with the product topology. This is a Polish space and a basis for its topology is given by the family of all sets of the form
\[
N(s_0,\dots,s_n) = \set{\alpha \in \baire}{\alpha(i) = s_i \quad \text{for all} \ i=0,\dots,n},
\]
where $s_0,\dots, s_n \in \nat$. Evidently these sets are clopen and therefore $\baire$ has a basis consisting of clopen sets, \ie it is a \emph{zero-dimensional} Polish space. The \emph{Cantor space} is $\{0,1\}^\nat \equiv \cantor$, and is a closed subspace of the Baire space. 

A \emph{continuous reduction} of a set $A \subseteq Z$ to a set $B \subseteq X$ is a continuous function $f: Z \to X$ such that $A = f^{-1}[B]$.  If a non-$\bolds^0_n$ set $A$ continuously reduces to a given set $B \subseteq X$, then $B$ cannot be a $\bolds^0_n$ subset $X$, since the class $\bolds^0_n$ is closed under continuous pre-images. This is a standard technique for showing that a given set $B$ is not $\bolds^0_n$: we start with a known non-$\bolds^0_n$ set $A$ and we show that it continuously reduces to $B$. In fact we usually show a slightly stronger property that is worth mentioning:

A set $P \subseteq X$ is $\boldp^0_n$\emph{-complete} if it is a $\boldp^0_n$ subset of $X$ and every $\boldp^0_n(Z)$ set $A$, where $Z$ is a zero-dimensional Polish space, continuously reduces to $P$. Analogously one defines the notion of $\bolds^0_n$-completeness, and it is clear that $P \subseteq X$ is $\boldp^0_n$-complete exactly when $X \setminus P$ is $\bolds^0_n$-complete.

It is easy to see that a $\boldp^0_n$-complete set $P \subseteq X$ cannot be $\bolds^0_n(X)$: since $\bolds^0_n(\cantor) \neq \boldp^0_n(\cantor)$ we can find some $A \in \boldp^0_n(\cantor) \setminus \bolds^0_n(\cantor)$; the  set $A$ continuously reduces to $P$, so if $P$ were $\bolds^0_n$ then $A$ (being the continuous pre-image of a $\bolds^0_n$ set) would be a $\bolds^0_n$ subset of $\cantor$, a contradiction.

Moreover it is clear that if a $\boldp^0_n$-complete set $P$ continuously reduces to a $\boldp^0_n$ set $B$ then $B$ is $\boldp^0_n$-complete as well. We can now state our main result.

\begin{theorem}
\label{theorem p complete}
For all $a,q$ with $0\leq a < q < \infty$ the intersection ${\textstyle \bigcap_{p > a}} \ell^p$ is a $\boldp^0_3$-complete subset of $\ell^q$. Moreover the continuous reductions can be chosen to take values in the closed unit ball of $\ell^q$. \smallskip

It follows that ${\textstyle \bigcap_{p > a}} \ell^p \not \in \bolds^0_3(\ell^q)$ and therefore ${\textstyle \bigcap_{p > a}} \ell^p$ is neither an $F_\sigma$ nor a $G_\delta$ subset of $\ell^q$.
\end{theorem}

Using the preceding result we can see that for every $\delta >0$ the continuous reductions in the latter can be chosen to take values in the closed $\delta$-ball of $\ell^q$ centered at $0$. This is because ${\textstyle \bigcap_{p > a}} \ell^p$ is a linear space and hence for every function $f: Z \to \ell^q$ that reduces some $Q$ to ${\textstyle \bigcap_{p > a}}\ell^p$ we will have
\[
z \in Q \iff f(z) \in {\textstyle \bigcap_{p > a}} \ell^p \iff \delta \cdot f(z) \in {\textstyle \bigcap_{p > a}}\ell^p,
\]
for all $z \in Z$. In other words the function $\delta \cdot f$ remains a reduction of $Q$ to ${\textstyle \bigcap_{p > a}}\ell^p$.

\begin{corollary}
\label{corollary intersection above b}
Suppose that $0\leq a < b < \infty$ and let $Y$ be one of the spaces  ${\textstyle \bigcap_{q > b}}\ell^q$, $c_0$ or $\ell^\infty$. Then the intersection ${\textstyle \bigcap_{p > a}} \ell^p$ is a $\boldp^0_3$-complete subset of $Y$, and in particular ${\textstyle \bigcap_{p > a}} \ell^p$ is neither an $F_\sigma$ nor a $G_\delta$ subset of $Y$.
\end{corollary}

\begin{proof}
From the the comments following the proof of Proposition 1 in \cite{nestoridis_a_project_about_chains_of_spaces} it follows that ${\textstyle \bigcap_{p > a}} \ell^p$ is a $\boldp^0_3$ subset of $Y$.

For the continuous reductions first we apply Theorem \ref{theorem p complete} with $\ell^b$ in the place of $\ell^q$. Next we notice that the identity $id: \ell^b \to Y$ is continuous. When $Y = c_0, \ell^\infty$ this is clear from the inequality $\norm{x}_\infty \leq \norm{x}_b$ for all $x \in \ell^b$. When $Y = {\textstyle \bigcap_{q > b}}\ell^q$ it is easy to verify that $d_{> b}(x,y) \leq \norm{x-y}_b$ for $b \geq 1$ and all $x, y \in \ell^b$, and also that $d_{> b}(x,y) \leq d_b(x,y)$ for $0 < b < 1$ and all $x,y \in \ell^b$ {\em with} $d_b(x,y) < 1$. 

\end{proof}

The notion of $\boldp^0_n$-completeness can be carried to the next level of the Borel hierarchy with the proper quantification and so from Theorem \ref{theorem p complete} we can get a natural example in the fourth level of the Borel hierarchy:

Given $0 \leq a < q < \infty$ the set $\ell^q \ \setminus \  {\textstyle \bigcap_{p > a}} \ell^p$ is a $\bolds^0_3$-complete subset of $\ell^q$. From \cite[Exercise 23.3]{kechris_classical_dst} the set of all sequences $(\vec{x}_{m})_{m \in \nat} \in \left(\ell^q\right)^\nat$ for which $\vec{x}_m \not \in {\textstyle \bigcap_{p > a}} \ell^p$ for all $m$, is a $\boldp^0_4$-complete subset of $\left(\ell^q\right)^\nat$.

We can in fact give the analogous result with $\ell^q$ in the place of $\left(\ell^q\right)^\nat$. Obviously we can view every sequence $(\vec{x}_{m})_{m \in \nat}$ of sequences as a double sequence $(x_{m,n})_{m,n \in \nat}$. Further we can put the latter in an infinite array whose $m$-row is $(x_{m,n})_{n \in \nat}$ and by a diagonal enumeration we can view $(x_{m,n})_{m,n \in \nat}$ as a single sequence. Moreover every single sequence can be identified with a double one using the preceding diagonal arrangement, therefore we can identify double sequences with the usual ones. It is clear that under this identification if $(x_{m,n})_{m,n \in \nat} \in \ell^q$ then for all $m$ the sequence $(x_{m,n})_{n \in \nat}$ is also a member of $\ell^q$ but the converse fails in general.

\begin{corollary}
\label{corollary s4 complete}
For all $a,q \in [0,\infty)$ with $a < q$ the following are $\bolds^0_4$- and $\boldp^0_4$-complete subsets of $\ell^q$ respectively:
\begin{align*}
A =& \ \set{(x_{m,n})_{m,n \in \nat} \in \ell^q}{(\exists m)[ \ (x_{m,n})_{n \in \nat} \ \in \ {\textstyle \bigcap_{p > a}} \ell^p\ ]}\\
B =& \ \set{(x_{m,n})_{m,n \in \nat} \in \ell^q}{(\forall m)[ \ (x_{m,n})_{n \in \nat} \ \not \in \ {\textstyle \bigcap_{p > a}} \ell^p\ ]},
\end{align*}
where a double sequence is identified with a usual one using a diagonal enumeration as above.\smallskip

The analogous statement holds if we replace $\ell^q$ with ${\textstyle \bigcap_{p > b}} \ell^p$ for $b > a$, $c_0$ or $\ell^\infty$.
\end{corollary} 

\begin{proof}
Since $A$ is the complement of $B$ it is enough to prove that the latter is a $\boldp^0_4$-complete subset of $\ell^q$. First we show that $B$ is a $\boldp^0_4$ set. 

For all $m$ we define $h_m: \ell^q \to \ell^q: (x_{k,n})_{k,n \in \nat} \mapsto (x_{m,n})_{\n}$; notice that $\sum_{\n} |x_{m,n} - y_{m,n}|^q \leq \norm{(x_{k,n})_{k,n} - (y_{k,n})_{k,n}}_q^q$ for $q \geq 1$, where $(x_{k,n})_{k,n},  \ (y_{k,n})_{k,n} \in \ell^p$. The similar assertion holds for $d_q$ if $q \in (0,1)$. Therefore $h_m$ is well-defined and $1$-Lispchitz. 

We also define the set $B_m = \set{(x_{k,n})_{k,n \in \nat} \in \ell^q}{(x_{m,n})_{n \in \nat} \not \in {\textstyle \bigcap_{p > a}} \ell^p}$ so that $B_m = h_m^{-1}[\ell^q \ \setminus \ {\textstyle \bigcap_{p > a}} \ell^p]$ for $m \in \nat$. The set $\ell^q \ \setminus \ {\textstyle \bigcap_{p > a}} \ell^p$ is $\bolds^0_3$ and hence $B_m$ is also $\bolds^0_3$ as the prei-mage of the former set under a continuous function. Hence $B = \textstyle\bigcap_{m \in \nat} B_m$ is a $\boldp^0_4$ subset of $\ell^q$. 

Suppose now that $Z$ is a zero-dimensional Polish space and that $P \subseteq Z$ is $\boldp^0_4$. Write $P = \textstyle \bigcap_{m \in \nat} P_m$ where $P_m$ is a $\bolds^0_3$ subset of $Z$. From Theorem \ref{theorem p complete} and its subsequent remarks there exists for each $m$ a continuous function $f_m: Z \to \ell^q$ such that $P_m = f_m^{-1}[\ell^q \ \setminus \ {\textstyle \bigcap_{p > a}} \ell^p]$ and $\norm{f_m(z)}_q^q \leq 2^{-m}$.

Define $f: Z \to \left(\ell^q\right)^\nat: f(z) = (f_m(z))_{m \in \nat}$. We write $f(z)$ as a double sequence $(x^z_{m,n})_{m,n \in \nat}$ so that $f_m(z) = (x^z_{m,n})_{n \in \nat}$ for all $m$. From the fact that $\norm{f_m(z)}_q^q \leq 2^{-m}$ it is evident that $f(z) \in \ell^q$. Moreover the function $f$ is continuous; this is because $\norm{f_m(z)}_q^q \leq 2^{-m}$ for all $z$ and the continuity of each $f_m$. Finally for all $z \in Z$ we have
\begin{align*}
z \in P 
\iff& \ (\forall m)[z \in P_m] \iff (\forall m)[f_m(z) \not \in {\textstyle \bigcap_{p > a}} \ell^p]\\
\iff& \ (\forall m)[(x^z_{m,n})_{n \in \nat} \not \in {\textstyle \bigcap_{p > a}} \ell^p] \iff f(z) \in B.
\end{align*}

\smallskip

Now we assume that $Y$ is one of ${\textstyle \bigcap_{q > b}} \ell^q$, $c_0$ or $\ell^\infty$. First we notice that for all $m$ the function $H_m: Y \to Y: (x_{k,n})_{k,n \in \nat} \mapsto (x_{m,n})_{\n}$ is well-defined and $1$-Lipschitz: this is clear when $Y = c_0, \ell^\infty$ since $H_m$ maps a sequence to a subsequence; when $Y = {\textstyle \bigcap_{q > b}} \ell^q$ the assertion follows as above with $\ell^q$ using also that the function $f(t) = t/(1+t)$, $t > -1$, is strictly increasing. Then the set $C = \textstyle \bigcap_{m \in \nat} H^{-1}_m[Y \ \setminus \ \bigcap_{p > a} \ell^p]$ is a $\boldp ^0_4$ subset of $Y$.

Regarding the continuous reductions to $C$ suppose that $P$ is a $\boldp^0_4$ subset of a zero-dimensional Polish space $Z$. We apply the first part of the proof with $b$ in the place of $q$ (if $Y = c_0, \ell^\infty$ we choose $b = q > a$) and we get a continuous function $f: Z \to \ell^b: z \mapsto (x^z_{m,n})_{m,n}$ such that 
\[
z \in P \iff (\forall m)[(x^z_{m,n})_{\n} \ \not \in \ \textstyle \bigcap_{p > a} \ell^p],
\]
for all $z$. Now as in the proof of Corollary \ref{corollary intersection above b} the function $id: \ell^b \to Y$ is continuous and $id \circ f: Z \to Y$ is the required reduction of $P$ to $C$.
\end{proof}

\section{The proof Theorem \ref{theorem p complete}}

We know from the comments following the proof of Proposition 1 in \cite{nestoridis_a_project_about_chains_of_spaces} that ${\textstyle \bigcap_{p > a}} \ell^p$ is a $\boldp^0_3$ subset of $\ell^q$. To carry out the proof it suffices to pick a $\boldp^0_3$-complete set $P_3 \subseteq \cantor$ and show that it reduces to ${\textstyle \bigcap_{p > a}} \ell^p$ via a continuous function $f: \cantor \to \ell^q$ which takes values in the closed unit ball of $\ell^q$. There is in fact a canonical choice for such $P_3$.

In the sequel we establish some terminology to be used in the proof.  A \emph{finite sequence} on a set $X$ is a function on $\set{i \in \nat}{i < n}$ to $X$, where $n \in \N$. We allow $n = 0$ in which case we mean the \emph{empty sequence} $\emptyset$. In general a finite sequence $\sigma$ will be denoted by $(\sigma(0),\dots,\sigma(n-1))$. The preceding $n$ is the \emph{length} of $\sigma$ and is denoted by $\lh(\sigma)$, so that $\sigma(i)$ is defined exactly when $i < \lh(\sigma)$, and $\sigma = (\sigma(0),\dots,\sigma(\lh(\sigma)-1))$.

By $X^{< \nat}$ we mean the set of all finite sequences of $X$. Given $\sigma, \tau \in X^{< \nat}$ we define the \emph{concatenation} $\cn{\sigma}{\tau}$ of $\sigma$ and $\tau$ to be the finite sequence that is obtained if we put $\sigma$ and $\tau$ together (starting with the former),
\[
\cn{\sigma}{\tau} = (\sigma(0),\dots, \sigma(\lh(\sigma-1)),\tau(0),\dots,\tau(\lh(\tau)-1)).
\]
We say that $\sigma$ is an \emph{initial segment} of $\tau$ or that $\tau$ \emph{extends} $\sigma$ and write $\sigma \sqsubseteq \tau$ if $\lh(\sigma) \leq \lh(\tau)$ and for all $i < \lh(\sigma)$ we have $\sigma(i) = \tau(i)$. We will write $\sigma \sqsubsetneq \tau$ when $\tau$ extends $\sigma$ properly.\smallskip

We fix the bijective function $\pair{\cdot}: \nat^2 \to \nat$ that ``moves diagonally upwards":
\[
\begin{tabular}{ccccc}
$0 = \pair{0,0}$&$2=\pair{0,1}$&$5=\pair{0,2}$&$9=\pair{0,3}$&$\cdots$\\[1ex]
$1 = \pair{1,0}$&$4=\pair{1,1}$&$8=\pair{1,2}$&$\cdots$&$\cdots$\\[1ex]
$3 = \pair{2,0}$&$7 = \pair{2,1}$&$\cdots$&$\cdots$&$\cdots$\\[1ex]
$6 = \pair{3,0}$&$\cdots$&$\cdots$&$\cdots$&$\cdots$\\[1ex]
\end{tabular}
\]

The canonical $\boldp^0_3$-complete set that we will use in our proof is
\[
P_3 = \set{\alpha \in \cantor}{(\forall i)(\exists j_0)(\forall j \geq j_0)[\alpha(\pair{i,j}) = 0]}.
\]
(This is up to homeomorphism the same set as the $P_3$ in \cite{kechris_classical_dst} p. 179, where it is proved that the latter set is $\boldp^0_3$-complete.)\smallskip

\begin{figure}
\caption{The diagonal arrangement of of $\sigma$ when $\lh(\sigma) = 8$}\bigskip
\label{figure A}
\begin{tabular}{ccccc}
Row $0$: & $\sigma(\pair{0,0})$&$\sigma(\pair{0,1})$&$\sigma(\pair{0,2})$&\\[2ex]
Row $1$: & $\sigma(\pair{1,0})$&$\sigma(\pair{1,1})$&&\\[2ex]
Row $2$: & $\sigma(\pair{2,0})$&$\sigma(\pair{2,1})$&&\\[2ex]
Row $3$: & $\sigma(\pair{3,0})$&&&\\
\end{tabular}
\end{figure}

Using the preceding pairing function we can arrange diagonally in an array every finite sequence $\sigma$ on a set $X$, see for example Figure \ref{figure A}.

The \emph{depth} $\d(\sigma)$ of a non-empty finite sequence $\sigma$ is the row with the largest number that is reached by $\sigma$ (we enumerate the rows starting with $0$), \ie
\[
d(\sigma) = \max\set{i}{(\exists k)[\sigma\pair{i,k} \neq \emptyset]}.
\]
According to Figure \ref{figure A} $\d(\sigma) = 3$ when $\lh(\sigma) = 8$. For technical reasons we define $\d(\emptyset) = -1$.

The \emph{level} $\l(\sigma)$ of a non-empty finite sequence $\sigma$ is the row where $\sigma$ obtains its last value, \ie
\begin{align*}
\l(\sigma) 
=& \ \text{\big\{the unique $i$ for which $\pair{i,k}$ is the greatest element}\\
& \ \hspace*{45mm} \text{in the domain of $\sigma$ for some $k$\big\}}.
\end{align*}
According to Figure \ref{figure A} $\l(\sigma) = 2$ when $\lh(\sigma) = 8$.

The following properties regarding the depth and the level of a non-empty sequence are easy to see:
\begin{align*}
\l(\sigma) \leq& \ \d(\sigma)\\
\sigma \sqsubseteq \sigma' \ \Longrightarrow& \ \d(\sigma) \leq \d(\sigma')\\
\d(\cn{\sigma}{(s)}) \leq& \ \d(\sigma)+1\\
\l(\sigma) = 0 \ \Longrightarrow& \ \l(\cn{\sigma}{(s)}) = \d(\cn{\sigma}{(s)}) = \d(\sigma)+1\\
\l(\sigma) > 0 \ \Longrightarrow& \ \l(\cn{\sigma}{(s)}) = \l(\sigma) - 1 \ \text{and} \  \d(\cn{\sigma}{(s)}) = \d(\sigma)
\end{align*}
where $\sigma \in X^{<\nat}$ is non-empty and $s \in \nat$.

For reasons of exposition we will make a slight abuse of the notation and denote $\sum_{n=0}^\infty |x_n|^p$ by $\norm{x}_p^p$ where $p > 0$ and $x = (x_n)_{\n} \in \ell^p$. Of course $\norm{\cdot}_p$ is a norm only when $p \geq 1$ in which case $\norm{x}_p^p$ is the $p$-th power of $\norm{x}_p$; if $p < 1$ the expression $\norm{x}_p^p$ is just another name for $d_p(x,0) = \sum_{n=0}^\infty |x_n|^p$.

The notions of depth and level will be utilized for $X = \{0,1\}$, but on the other hand we will also be dealing with finite sequences of real numbers. We will regularly identify a finite sequence $u \in \R^{< \nat}$ with the infinite one $\cn{u}{\vec{0}} = \cn{u}{(0,0,\dots)} \in c_{00}$, so that when we write $\norm{u-v}_p^p$ for $u,v \in \R^{< \nat}$ we mean $\norm{\cn{u}{\vec{0}} - \cn{v}{\vec{0}}}_p^p$. It is then clear that
\[
\norm{u}_p^p = \sum_{n < \lh(u)} |u(n)|^p.
\]
Notice that
\[
\norm{\cn{u}{v} - u}_p^p = \norm{v}_p^p.
\]
\smallskip

\emph{Claim:} Suppose that $q > p_0 > p_1 > \dots > p_k > p_{k+1} > 0$, $r_0, \dots, r_k, M, \ep > 0$ and that $u \in \R^{< \nat}$ is such that $\norm{u}^{p_i}_{p_i} < r_i$ for all $i=0,1,\dots,k$. Then there exists a non-empty finite sequence $v \in [0,\infty)^{<\nat}$ such that
\begin{align*}
\norm{v}_q^q <& \ \ep,\\
\norm{\cn{u}{v}}^{p_i}_{p_i} <& \ r_i, \quad i=0,1,\dots,k\\
\norm{\cn{u}{v}}_{p_{k+1}}^{p_{k+1}} >& \ M. 
\end{align*}\smallskip

\emph{Proof of the Claim.}

Put $\delta = \min \set{r_i - \norm{u}_{p_i}^{p_i}}{i=0,\dots,k} > 0$ and consider a sequence $(x_n)_{\n}$ of non-negative real numbers which is a member of $\ell^{p_k}  \setminus \ell^{p_{k+1}}$. 

Then there is $n_0$ such that $\sum_{n = n_0}^\infty x_n^{p_k} < \min\{\delta,\ep\}$ and $x_n < 1$ for all $n \geq n_0$. It follows that
\[
\sum_{n = n_0}^\infty x_n^{p_i} \leq \sum_{n = n_0}^\infty x_n^{p_k} < \delta
\]
for all $i=0,\dots,k$. Further 
\[
\sum_{n=n_0}^\infty x_n^q \leq \sum_{n=n_0}^\infty x_n^{p_k} < \ep.
\]
Since $(x_n)_{\n} \not \in \ell^{p_{k+1}}$ we have $\sum_{n = n_0}^\infty x_n^{p_{k+1}} = \infty$. Therefore there is some $n_1 \geq n_0$ such that $\sum_{n=n_0}^{n_1} x_n^{p_{k+1}} > M$.

Take $v = (x_{n_0},\dots,x_{n_1})$. Clearly $\norm{v}_q^q = \sum_{n = n_0}^{n_1} x_n^q \leq \sum_{n=n_0}^\infty x_n^q < \ep$. For all $i=0,\dots,k$ we have
\begin{align*}
\norm{\cn{u}{v}}_{p_i}^{p_i} 
=& \ \sum_{n < \lh(u)} |u(n)|^{p_i} + \sum_{n=n_0}^{n_1} x_n^{p_i}\\
\leq& \ \norm{u}_{p_i}^{p_i} + \sum_{n=n_0}^{\infty} x_n^{p_i}\\
<& \ \norm{u}_{p_i}^{p_i} + \delta_0\\
\leq& \ r_i \ .
\end{align*}
Moreover
\[
\norm{\cn{u}{v}}_{p_{k+1}}^{p_{k+1}} \geq \sum_{n=n_0}^{n_1} x_n^{p_{k+1}} > M.
\]
This concludes the proof of the claim.

\bigskip

In the sequel we fix a sequence of real numbers $(p_i)_{i \in \nat}$ with $p_i \downarrow a$, so that ${\textstyle \bigcap_{p > a}} \ell^p = \cap_{i \in \nat} \ell^{p_i}$, and with $p_0 < q$. For example $(p_i)_{i \in \nat}$ could be a shift of $(p^a_i)_{i \in \nat}$.\bigskip

{\em The main construction.} We show that for every non-empty $\sigma \in 2^{< \nat}$ there is a non-empty $\varphi(\sigma) \in [0,+\infty)^{<\nat}$ and natural numbers $M_i(\sigma)$, $0 \leq i \leq \d(\sigma)$ with the following properties:
\begin{align}
\label{equation proposition p complete A}
\sigma' \sqsubsetneq \sigma \ \Longrightarrow& \ \ \varphi(\sigma') \sqsubsetneq \varphi(\sigma)\\
\label{equation proposition p complete B}
& \hspace*{-17mm} \norm{\varphi(\cn{\sigma}{(s)}) - \varphi(\sigma)}_q^q < \ 2^{-(\lh(\sigma)+1)} \quad
\text{and} \quad \norm{\varphi((s))}_q^q < 2^{-1}, \quad \quad s=0,1\\
\label{equation proposition p complete C}
\norm{\varphi(\sigma)}_{p_i}^{p_i} <& \ M_i(\sigma), \quad i=0,\dots,\d(\sigma)\\
\label{equation proposition p complete E}
M_i(\cn{\sigma}{(0)}) =& \ M_i(\sigma), \quad i=0,\dots,\d(\sigma)\\
\label{equation proposition p complete F}
M_i(\cn{\sigma}{(1)}) =& \ M_i(\sigma), \quad i=0,\dots,\l(\cn{\sigma}{(1)})-1 \quad \text{if $\l(\cn{\sigma}{(1)}) > 0$}\\
\label{equation proposition p complete G}
& \hspace*{-17mm} \norm{\varphi(\cn{\sigma}{(1)})}_{p_i}^{p_i} > \lh(\sigma)+1, \quad \text{where} \ i = \l(\cn{\sigma}{(1))})
\end{align}

As to (\ref{equation proposition p complete E}) notice that $i \leq \d(\sigma) \leq \d(\cn{\sigma}{(0)})$, so that $M_i(\cn{\sigma}{(0)})$ is defined. Similarly in (\ref{equation proposition p complete F}) we have $\l(\cn{\sigma}{(1)}) - 1 \leq \d(\cn{\sigma}{(1)}) - 1 \leq \d(\sigma)+1 - 1 = \d(\sigma)$, so that $M_i(\sigma)$ and $M_i(\cn{\sigma}{(1)})$ are both defined.

The idea in order to show the above assertions is roughly as follows. When we extend $\sigma$ by $0$ we also extend $\varphi(\sigma)$ by $0$; the norms remain the same and hence the $M_i$'s can remain the same for $i \leq \d(\sigma)$. We might need to add one more $M_i$, namely $M_{\d(\sigma)+1}(\cn{\sigma}{(0)})$ if the depth of $\cn{\sigma}{(0)}$ increases by $1$ from $\d(\sigma)$, but this poses no problems. The interesting case is when we extend by $1$. Then we extend $\varphi(\sigma)$ so that for all $i < \l(\cn{\sigma}{(1)})$ the $p_i$-norm remains below $M_i(\sigma)$ and hence we can take $M_i(\cn{\sigma}{(1)})$ to be the same as $M_i(\sigma)$; on the other hand we make a substantial increase on the $p_i$-norm for $i = \l(\cn{\sigma}{(1)})$. This is possible from the preceding Claim. The $M_i$'s for $i = \l(\cn{\sigma}{(1)}), \dots, \d\cn{\sigma}{(1)})$ are easily arranged.\smallskip

Formally we define functions 
\[
\varphi: 2^{<\nat} \to [0,\infty)^{<\nat} \quad \text{and} \quad \psi: 2^{<\nat} \to \nat^{<\nat}: \psi(\sigma) = (M_0(\sigma),\dots,M_{\d(\sigma)}(\sigma))
\]
which satisfy the required properties. The definition is done by recursion on $\lh(\sigma)$ starting with $\lh(\sigma) = 0$.

Put $\varphi(\emptyset) = \psi(\emptyset) = \emptyset$. Assume that for some $n \in \nat$ we have the following:
\begin{list}{}{}
\item[$(D^\ast)$]: $\varphi(\sigma)$, $\psi(\sigma)$ are defined and $\lh(\psi(\sigma)) = \d(\sigma)+1$ for all $\sigma$ with length at most $n$,
\item[$(\ref{equation proposition p complete A}^\ast)$]: property (\ref{equation proposition p complete A}) holds for all $\sigma,\sigma'$ with length at most $n$,
\item[$(\ref{equation proposition p complete B}^\ast)$]: the first part of (\ref{equation proposition p complete B}) holds for all $\sigma$ with $\lh(\cn{\sigma}{(s)}) \leq n$ and if $n = 1$ the second part of (\ref{equation proposition p complete B}) holds as well,
\item[$(\ref{equation proposition p complete C}^\ast)$]: property (\ref{equation proposition p complete C}) holds for all $\sigma \neq \emptyset$ with length at most $n$,
\item[$(\ref{equation proposition p complete E}^\ast)$]: property (\ref{equation proposition p complete E}) holds for all non-empty $\sigma$ with $\lh(\cn{\sigma}{(0)}) \leq n$,
\item[$(\ref{equation proposition p complete F}^\ast)$]: property (\ref{equation proposition p complete F}) holds for all non-empty $\sigma$ with $\lh(\cn{\sigma}{(1)}) \leq n$, and
\item[$(\ref{equation proposition p complete G}^\ast)$]: property (\ref{equation proposition p complete G}) holds for all -possibly empty- $\sigma$ with $\lh(\cn{\sigma}{(1)}) \leq n$.
\end{list}
We show now that the properties $(D^\ast)$, $(\ref{equation proposition p complete A}^\ast)$-$(\ref{equation proposition p complete G}^\ast)$ hold for $n+1$. First we define $\varphi(\tau)$ and $\psi(\tau)$ when $\lh(\tau) = n+1$. There are two cases: a) $\tau = \cn{\sigma}{(0)}$ and b) $\tau = \cn{\sigma}{(1)}$ for some $\sigma$ with $\lh(\sigma) = n$.

In the case of a) we define $\varphi(\tau) = \varphi(\cn{\sigma}{(0)}) = \cn{\varphi(\sigma)}{(0)}$. 

As for $\psi(\sigma)$ suppose first that $\sigma \neq \emptyset$ so that $\d(\sigma) \geq 0$. We define $M_i(\tau) = M_i(\cn{\sigma}{(0)}) = M_i(\sigma)$ for $i=0,\dots,\d(\sigma)$. If $\d(\tau) = \d(\sigma)$ this completes the definition. Else $\d(\tau) = \d(\sigma) + 1$ and we define further $M_{\d(\tau)}(\tau)$ to be the least natural greater than $\norm{\varphi(\tau)}_{p_{\d(\tau)}}^{p_{\d(\tau)}}$. If $\sigma = \emptyset$ then $\d(\tau) = \d((0)) = 0$ and we define $M_0(\tau) = 1$.\smallskip

In the case of b) we apply the Claim from above. Assume first that $\l(\tau) = \l(\cn{\sigma}{(1)}) > 0$, so in particular $\sigma \neq \emptyset$. Put $k = \l(\tau) - 1 = \l(\cn{\sigma}{(1)}) - 1 \leq \d(\sigma)$. From the induction hypothesis $\norm{\varphi(\sigma)}_{p_i}^{p_i} < M_i(\sigma)$ for all $i=0,\dots,k$. So from the preceding claim there exists some $v \in [0,\infty)^{<\nat}$ such that $\norm{v}^q_q < 2^{-(\lh(\sigma)+1)}$, $\norm{\cn{\varphi(\sigma)}{v}}_{p_i}^{p_i}< M_i(\sigma)$ for all $i=0,\dots,k$ and $\norm{\cn{\varphi(\sigma)}{v}}_{p_{k+1}}^{p_{k+1}} > \lh(\sigma) + 1$. We define
\[
\varphi(\tau) = \varphi(\cn{\sigma}{(1)}) = \cn{\varphi(\sigma)}{v} \quad \text{and} \quad M_i(\tau) = M_i(\cn{\sigma}{(1)}) = M_i(\sigma), \quad i=0,\dots,k.
\]
We define further $M_{i}(\tau)$ to be least natural greater than $\norm{\cn{\varphi(\sigma)}{v}}_{p_i}^{p_i}$ for $i=\l(\tau)$,$\dots$,$\d(\tau)$.

The remaining sub-case is when $\l(\tau) = \l(\cn{\sigma}{(1)}) = 0$. (This includes the case $\sigma = \emptyset$.) According to the Claim above there exists some $v \in [0,\infty)^{<\nat}$ such that $\norm{v}^{q}_q < 2^{-(\lh(\sigma)+1)}$ and $\norm{\cn{\varphi(\sigma)}{v}}_{p_0}^{p_0} > \lh(\sigma) + 1$. We define
\[
\varphi(\tau) = \varphi(\cn{\sigma}{(1)}) = \cn{\varphi(\sigma)}{v} 
\]
and $M_i(\tau)$ to be the least natural greater than $\norm{\cn{\varphi(\sigma)}{v}}_{p_i}^{p_i}$ for $i=0$,$\dots$,$\d(\tau)$. This settles $(D^\ast)$ for $n+1$.\smallskip

It is not hard to verify that the properties $(\ref{equation proposition p complete A}^\ast)$-$(\ref{equation proposition p complete G}^\ast)$ hold for $n+1$ and the inductive step is done.\smallskip

The properties $(\ref{equation proposition p complete A})$-$(\ref{equation proposition p complete G})$ are immediate from $(\ref{equation proposition p complete A}^\ast)$-$(\ref{equation proposition p complete G}^\ast)$ for sufficiently large $n$.\smallskip

\emph{The definition of the reduction.} We take the function $\varphi: 2^{< \nat} \to [0,+\infty)^{< \nat}$ and the naturals $M_i(\sigma)$, $0 \leq i \leq \d(\sigma)$, $\sigma \in 2^{<\nat}$, as above. From (\ref{equation proposition p complete A}) the sequences $\varphi((\alpha(0),\dots,\alpha(k)))$ for $k \in \nat$ are compatible and their union forms an infinite sequence of real numbers.

We define $f:\cantor \to \R^\nat$ such that $f(\alpha)$ is the unique infinite sequence of real numbers that is formed by $\varphi((\alpha(0),\dots,\alpha(k)))$ for $k \in \nat$. In other words 
\[
f(\alpha)(n) = \varphi((\alpha(0),\dots,\alpha(k)))(n)
\]
for all large $k$, where $\n$.\smallskip

It is helpful to take the sequence $(u^\alpha_k)_{k \in \nat}$ of members of $[0,+\infty)^{<\nat}$ with $u_0^\alpha = \varphi(\alpha(0))$ and $\cn{\varphi((\alpha(0),\dots,\alpha(k)))}{u_{k+1}^\alpha} = \varphi((\alpha(0),\dots,\alpha(k),\alpha(k+1)))$, so that
\[
f(\alpha) = u^\alpha_0 \consymbol u^\alpha_1 \consymbol \dots \consymbol u^\alpha_k \consymbol \dots
\]
Moreover it is clear that
\[
\varphi((\alpha(0),\dots,\alpha(k))) = u^\alpha_0 \consymbol \dots \consymbol u^\alpha_{k}.
\]

So if $\sigma = (\alpha(0),\dots,\alpha(k))$,
\[
\norm{\varphi(\cn{\sigma}{(\alpha(k+1))}) - \varphi(\sigma)}_p^p = \norm{\cn{\varphi(\sigma)}{u^\alpha_{k+1}} - \varphi(\sigma)}_p^p =  \sum_{n < \lh(u^\alpha_{k+1})}| u^\alpha_{k+1}(n)|^p  
\]
for all $p \geq 1$; in particular from (\ref{equation proposition p complete B}) we have
\begin{align}
\label{equation proposition p complete H}
\sum_{n < \lh(u^\alpha_{k})}| u^\alpha_{k}(n)|^q  < 2^{-(k+1)}
\end{align}
for all $\alpha \in \cantor$ and all $k \in \nat$. (The case $k=0$ follows directly from the second inequality of (\ref{equation proposition p complete B})).

\smallskip

\emph{The function $f$ takes values in the closed unit ball of $\ell^q$.} For all $\alpha \in \cantor$ and all $m \in \nat$ there is $k \in \nat$ such that the sequence $u^\alpha_0 \consymbol \dots \consymbol u^\alpha_k$ extends $(f(\alpha)(0),\dots,f(\alpha)(n))$; so from (\ref{equation proposition p complete H}),
\[
\sum_{n=0}^m |f(\alpha)(n)|^q \leq \sum_{t=0}^k \sum_{n < \lh(u^\alpha_{t})} |u^\alpha_t(n)|^q  \leq \sum_{t=0}^k 2^{-(t+1)} \leq 1.
\]
It follows that $f(\alpha) \in \ell^q$ and $\norm{f(\alpha)}_q^q \leq 1$.\smallskip

\emph{The function $f$ is continuous.} Applying (\ref{equation proposition p complete H}) again we observe that
\[
\norm{f(\alpha) - u^\alpha_0 \consymbol \dots \consymbol u^\alpha_k}_q^q = \sum_{t=k+1}^\infty \sum_{n < \lh(u^\alpha_t)} |u^\alpha_t(n)|^q \leq \sum_{t=k+1}^\infty 2^{-(t+1)} = 2^{-(k+1)}
\]
for all $\alpha \in \cantor$ and all $k \in \nat$.

Given $\alpha \in \cantor$ and $\ep > 0$ we choose $k$ large enough so that $2^{-(k+1)/q} < \ep/2$ if $q \geq 1$ and $2^{-(k+1)} < \ep/2$ if $0 < q < 1$. If $\beta$ agrees with $\alpha$ up to $k$, \ie if $\beta \in N(\alpha(0),\dots,\alpha(k))$ we have $u^\alpha_i = u^\beta_i$ for all $i = 0,\dots,k$.

Case $q \geq 1$:
\begin{align*}
\norm{f(\alpha) - f(\beta)}_q 
\leq& \ \norm{f(\alpha) - u^\alpha_0 \consymbol \dots \consymbol u^\alpha_k}_q + \norm{u^\beta_0 \consymbol \dots \consymbol u^\beta_k - f(\beta)}_q\\
<& \ 2^{-(k+1)/q} + 2^{-(k+1)/q} < \ep/2 + \ep/2 = \ep.
\end{align*}

Case $0 < q < 1$:

\begin{align*}
d_q(f(\alpha),f(\beta)) 
\leq& \ d_q(f(\alpha),u^\alpha_0 \consymbol \dots \consymbol u^\alpha_k \consymbol \vec{0}) + d_q(u^\beta_0 \consymbol \dots u^\beta_k \consymbol \vec{0},f(\beta))\\
=& \ \norm{f(\alpha) - u^\alpha_0 \consymbol \dots \consymbol u^\alpha_k}_q^q + \norm{u^\beta_0 \consymbol \dots \consymbol u^\beta_k - f(\beta)}_q^q\\
<& \ 2^{-(k+1)} + 2^{-(k+1)} < \ep/2 + \ep/2 = \ep.
\end{align*}

\emph{The function $f$ is a reduction between the required sets.} We fix an $\alpha \in \cantor$. We need to show that
\[
\alpha \in P_3 \iff f(\alpha) \in {\textstyle \bigcap_{p > a}} \ell^p.
\]
Assume first that $\alpha \not \in P_3$, so that for some $i \in \nat$ we have $\alpha(\pair{i,j})=1$ for infinitely many $j$. We can then find a strictly increasing sequence $(j_m)_{m \in \nat}$ of positive naturals such that $\alpha(\pair{i,j_m}) = 1$ for all $m$. We fix $m \in \nat$ and we put $\sigma = (\alpha(0),\dots,\alpha(\pair{i,j_m}-1))$, so that $(\alpha(0),\dots,\alpha(\pair{i,j_m})) = \cn{\sigma}{(1)}$. Clearly $\l(\cn{\sigma}{(1)}) = i$, so from (\ref{equation proposition p complete G}),
\begin{align*}
\sum_{n=0}^\infty |f(\alpha)(n)|^{p_i} \geq \sum_{n < \lh(\varphi(\cn{\sigma}{(1)}))}|\varphi(\cn{\sigma}{(1)})(n)|^{p_i} = \norm{\varphi(\cn{\sigma}{(1)})}_{p_i}^{p_i} > \pair{i,j_m} + 1.
\end{align*}
Since $\lim_{m \to \infty} \pair{i,j_m} = \infty$ we have that $f(\alpha) \not \in \ell^{p_i}$, and so $f(\alpha) \not \in {\textstyle \bigcap_{p > a}} \ell^p$.\smallskip

\begin{figure}
\caption{The diagonal arrangement of $\alpha \in P_3$}\bigskip
\label{figure B}
\begin{tabular}{ccccccccc}
\hphantom{Row $j$:} &&&Col. $j_0$&&&&\\[1ex]
\hphantom{Row $0$:} \ \ $\alpha(0)$&$\alpha(2)$&\dots&\tm[1.8ex]{E}$0$&$0$&\dots&$0$&\dots\\[1ex]
\hphantom{Row $1$:} \ \ $\alpha(1)$&\underline{\hphantom{000}} &\dots&$0$&$0$&\dots&$0$&\dots\\[1ex]
\hphantom{Row $2$:} \ \ \vdots& &\hphantom{\dots}&&&\vdots&&\\[1ex]
Row $i$: \ \ $\alpha(\pair{i,0})$&\underline{\hphantom{000}} &\dots&$\tm[-0.9ex]{B}0$&$0$&\dots&$0$&\dots\tm[-0.9ex]{F}\\[1ex]
\hphantom{Row $j$:} \ \ \ \ \raisebox{1ex}{\footnotesize{\vdots}}\hspace*{-3mm}\underline{\hphantom{000}}&\underline{\hphantom{000}} &\dots&\raisebox{3ex}{\scriptsize{$\sigma_0$}}$\hspace*{-3mm}\underline{\hphantom{000}}$&$\underline{\hphantom{000}}$&\dots&&\\[1ex]
\hphantom{Row $j$:} \ \ \underline{\hphantom{000}}&\underline{\hphantom{000}} &\dots&$\underline{\hphantom{000}}$&$\underline{\hphantom{000}}$&&&\\[1ex]
\hphantom{Row $j$:} \ \ \tm[0ex]{A} \underline{\hphantom{000}}&\tm[1.0ex]{D}1 &\hspace*{-4.5mm}\raisebox{1.8ex}{\scriptsize{$\sigma_0\consymbol\dots\consymbol(1)$}}\hspace*{-11.5mm}\dots&$\underline{\hphantom{000}}$&&&&\\[1ex]
\hphantom{Row $j$:} \ \ \tm[-0.3ex]{C}\underline{\hphantom{000}}&\underline{\hphantom{000}}&\dots&&&&&
\end{tabular}
\tikz[remember picture, overlay]\draw[->] (A)--(B);
\tikz[remember picture, overlay]\draw[->] (C)--(D);
\tikz[remember picture, overlay]\draw (E)--(B);
\tikz[remember picture, overlay]\draw (B)--(F);
\end{figure}

Finally assume that $\alpha \in P_3$ and we show that $f(\alpha) \in \ell^{p_i}$ for all $i$. We take some $i \in \nat$; since $\alpha \in P_3$ there some $j_0$ such that for all $j \geq j_0$ and all $i' \leq i$ we have $\alpha(\pair{i',j}) = 0$.

We put $\sigma_0 = (\alpha(0),\dots, \alpha(\pair{i,j_0}))$. Evidently $i = \l(\sigma_0) \leq \d(\sigma_0)$ and so $M_i(\tau)$ is defined for all $\tau$ extending $\sigma_0$. The idea is to show that $M_i(\sigma_0)$ controls $M_i(\tau)$ as $\tau$ extends $\sigma_0$ {\em along} $\alpha$. This is clear by (\ref{equation proposition p complete E}) when adding $0$; when adding $1$, due to the choice of $j_0$, the level of our sequence must be below $i$, see Figure \ref{figure B}, and therefore we are covered by (\ref{equation proposition p complete F}).

To make the latter precise we claim that 
\begin{align}
\label{equation proposition p complete I}
\text{$M_i(\sigma_0) = M_i(\tau)$ for all $\tau = (\alpha(0),\dots,\alpha(k))$, where $k \geq \pair{i,j_0}$.}
\end{align}
We prove this by induction on $k$. If $k = \pair{i,j_0}$ then $\tau = \sigma_0$ and the assertion is trivial. Assume that (\ref{equation proposition p complete I}) holds for some $k \geq \pair{i,j_0}$. We put $\sigma = (\alpha(0),\dots,\alpha(k))$ and $\tau = (\alpha(0),\dots,\alpha(k),\alpha(k+1))$, so that $\tau = \cn{\sigma}{(\alpha(k+1))}$ and $\sigma_0 \sqsubseteq \sigma$.

From the induction hypothesis $M_i(\sigma) = M_i(\sigma_0)$. If $\alpha(k+1) = 0$ then $\tau = \cn{\sigma}{(0)}$ and from (\ref{equation proposition p complete E}), since $i \leq \d(\sigma_0)\leq \d(\sigma)$ we have $M_i(\tau) = M_i(\sigma) = M_i(\sigma_0)$. Now we assume $\alpha(k+1) = 1$ and so $\tau = \cn{\sigma}{(1)}$. Since $\alpha(\pair{i',j}) = 0$ for all $j \geq j_0$ and all $i' \leq i$ the level of the finite sequence $\tau$ is at least $i+1$ (see Figure \ref{figure B}). Hence $\l(\cn{\sigma}{(1)}) - 1 = \l(\tau) - 1 \geq i \geq 0$, so (\ref{equation proposition p complete F}) is applicable to $i$ and $\tau = \cn{\sigma}{(1)}$. From the latter it follows $M_i(\tau) = M_i(\cn{\sigma}{(1)}) = M_i(\sigma) = M_i(\sigma_0)$ and the induction step is complete.\smallskip

From (\ref{equation proposition p complete C}) and (\ref{equation proposition p complete I}) we have for all $k \geq \pair{i,j_0}$,
\begin{align*}
\sum_{t=0}^k \sum_{n < \lh(u^\alpha_t)} |u^\alpha_t(n)|^{p_i} =& \ \norm{u^\alpha_0\consymbol \dots\consymbol u^\alpha_k}_{p_i}^{p_i} = \norm{\varphi((\alpha(0),\dots,\alpha(k)))}_{p_i}^{p_i}\\
<& \ M_i((\alpha(0),\dots,\alpha(k))) = M_i(\sigma_0).
\end{align*}
It follows that $\sum_{n=0}^{\infty} |f(\alpha)(n)|^{p_i} \leq M_i(\sigma_0)$ and therefore $f(\alpha) \in \ell^{p_i}$.
\enlargethispage{4pt}


\end{document}